\documentclass[12pt,a4paper,psamsfonts]{amsart}
\usepackage{fancyhdr}
\usepackage{appendix}
\usepackage{amssymb,amscd,amsxtra,calc}
\usepackage{mathrsfs}
\usepackage{amsmath}
\usepackage{cmmib57}
\usepackage{multirow}
\usepackage{verbatim}
\usepackage[all]{xy}
\usepackage[colorlinks,linkcolor=blue,anchorcolor=blue,citecolor=green]{hyperref}
\theoremstyle{plain}
    \newtheorem{thm}{Theorem}[section]
    
    \newtheorem{claim}[thm]{Claim}

    \newtheorem{lemma}[thm]{Lemma}
    \newtheorem{proposition}[thm]{Proposition}
    
    \newtheorem{theorem}[thm]{Theorem}

\theoremstyle{definition}
    \newtheorem{definition}[thm]{Definition}
    
    \newtheorem{notation}[thm]{Notation}
    \newtheorem*{notation*}{Notation and Terminology}
    \newtheorem{remark}[thm]{Remark}
       \newtheorem{problem}[thm]{Problem}
    
\theoremstyle{remark}

\makeatletter

\newcommand{\Rmnum}[1]{\expandafter\@slowromancap\romannumeral #1@}
\makeatother
\begin{document}

\bibliographystyle{plain}
\title[maximal rank]
%{Polarized or int-amplified endomorphisms of normal projective varieties in arbitrary characteristic
%}
{Existence of the equivariant minimal model program for compact K\"ahler threefolds with the action of an abelian group of maximal  rank% a group action---An addendum to ``Compact K\"ahler threefolds with the action of an abelian group of maximal rank''
}

\author{Guolei Zhong}

\address
{Center for Complex Geometry, Institute for Basic Science (IBS), 55 EXPO-ro,  Yuseong-gu, Daejeon, 34126, Republic of Korea.}
\email{zhongguolei@u.nus.edu, guolei@ibs.re.kr}

\begin{abstract}
%In this addendum to the paper \cite{Zho22},
Let $X$ be a  $\mathbb{Q}$-factorial compact K\"ahler klt threefold  admitting an action of a free abelian group $G$, which is  of positive entropy and of maximal  rank. 
After running the $G$-equivariant log minimal model program, we show that such $X$ is either rationally connected or  bimeromorphic to a $Q$-complex torus. 
In particular, we fix an issue in the  proof of our previous paper \cite[Theorem 1.3]{Zho22}. 
%which results from a temporary gap in the proof of the abundance theorem \cite[Theorem 1.1]{CHP16} as recently claimed in \cite{CHP21}.
\end{abstract}


\maketitle

\section{Introduction}
We work over the field $\mathbb{C}$ of complex numbers. 
Let $X$ be a compact K\"ahler space with only rational singularities.
Given an automorphism $f\in\textup{Aut}(X)$, the \textit{first dynamical degree} $d_1(f)$ is defined as the spectral radius $\rho(f^*)$ of the induced linear operation $f^*$ on the Bott-Chern cohomology space $H_{\textup{BC}}^{1,1}(X)$ (cf.~\cite[Definition 4.6.2]{BEG13}).  
We say that $f$ is of \textit{positive entropy} if $d_1(f)>1$.  
Note that the definition for $d_1(f)$ here coincides with the usual one when $X$ is smooth (cf.~the arXiv version  \cite[Remark 1.2]{Zho22}). 

A subgroup $G\subseteq\textup{Aut}(X)$ is of \textit{positive entropy}, if every non-trivial element of $G$ is of positive entropy. 
Due to the existence of the $G$-equivariant resolution (cf.~e.g.~\cite[Theorem 2.0.1]{Wlo09}), 
 Dinh and Sibony proved that the rank of a free abelian group $G$, which is of positive entropy, is no more than $\dim X-1$; see \cite[Theorem I]{DS04}. 
We refer readers to \cite[Theorem 1.1]{Zha09} for the  extension to the solvable group case by Zhang so as to prove a theorem of Tits type for automorphism groups of compact K\"ahler manifolds.% (cf.~the survey paper \cite{Din12}).

As mentioned in \cite{DS04}, it is interesting to consider the extremal case when $G$ achieves the maximal rank $\dim X-1$.
In other words, we would like to describe the geometric property from the dynamical viewpoints.

\begin{problem}[{\cite[Problem 1.5]{Din12}; cf.~\cite[Problem 1.1]{Zho22}}]\label{pro-main}
Classify compact K\"ahler manifolds $X$ of dimension $n\ge 3$ admitting
a free abelian group $G$ of automorphisms of rank $n-1$ which is of positive entropy.
\end{problem}
For projective varieties, Problem \ref{pro-main} has been intensively studied by Zhang in his series of papers (cf.~\cite{Zha09}, \cite{Zha16}). 
In the analytic case, we showed in \cite{Zho22} that, a $\mathbb{Q}$-factorial compact K\"ahler terminal threefold $X$ admitting an action of a free abelian group $G$, which is of positive entropy and of maximal rank, is either rationally connected or  bimeromorphic to a \textit{$Q$-complex torus}, i.e., there is a finite surjective morphism from a complex torus $T\to X$, which is \'etale in codimension one.  
In the proof of that main result \cite[Theorem 1.3]{Zho22} however, we applied the abundance theorem for  K\"ahler threefolds; see \cite[Theorem 1.1]{CHP16}.  
As recently claimed in \cite{CHP21},  the proof of  \cite[Theorem 1.1]{CHP16} for the case when the Kodaira dimension $\kappa(X)=0$ and the algebraic dimension $a(X)=0$ seems incomplete.   

In this note, we will  provide an independent proof of \cite[theorem 1.3]{Zho22} by constructing the $G$-equivariant minimal model program ($G$-MMP for short) for K\"ahler threefolds  so as to fix the resulting issue in our previous paper (cf.~\cite{Zha16} for the $G$-MMP on projective varieties).  
Moreover, thanks to the recent inspiring progress of the minimal model theory on K\"ahler threefolds (\cite{DO22}, \cite{DH22}; cf.~\cite{HP15}, \cite{HP16}, \cite{CHP16} and \cite{CHP21}), we can strengthen \cite[Theorem 1.3]{Zho22} by weakening the condition of $X$ having terminal singularities to that of $X$ having only klt singularities (cf.~Theorem \ref{thm-main}).

We refer readers to \cite[Introduction]{Zho22} and the references therein for more backgrounds. 
In this note, we will focus on a clean proof of the existence of the $G$-MMP on K\"ahler threefolds (cf.~Remark \ref{rmk-difference}). 
Let us begin with the following assumption. 

\setlength\parskip{8pt}\par \vskip 0.2pc 

\noindent
\textbf{\hypertarget{Hypothesis}{Hypothesis}:} Let $(X,D)$ be  a normal $\mathbb{Q}$-factorial compact K\"ahler threefold pair with only klt singularities (where the boundary $D$ is some effective $\mathbb{Q}$-divisor). 
Suppose that there is an action of a free abelian group $G\cong\mathbb{Z}^2\subseteq \textup{Aut}(X)$, which is of positive entropy and of maximal rank. 
Suppose further that every irreducible component of the support of $D$ is $G$-periodic.

The  theorem below is our main result of this note. 
 \setlength\parskip{0pt}
 \par \vskip 0pc \noindent
 
\begin{theorem}\label{thm-main}
Let $(X,D,G)$ be a pair satisfying \hyperlink{Hypothesis}{Hypothesis}. 
If $X$ is not rationally connected, then with $G$ replaced by a finite-index subgroup, the following assertions hold.
\begin{enumerate}
\item There is a $G$-equivariant bimeromorphic map $X\dashrightarrow Z$ (i.e., $G$ descends to a biregular action on $Z$) such that $Z$ is a $Q$-complex torus. 
\item $Z\cong T/H$ where $H$ is  a finite group acting freely outside a finite set of a complex $3$-torus $T$. 
Moreover, the quotient morphism $T\to Z$ is also $G$-equivariant.
\item There is no positive dimensional $G$-periodic proper subvariety of $Z$; in particular, the bimeromorphic map $X\dashrightarrow Z$ in (1) is holomorphic. 
\end{enumerate}
\end{theorem}

%Throughout the proof of Theorem \ref{thm-main}, we utilized the base point free theorem which was recently proved  in \cite[Theorem 1.7]{DH22}.
%\begin{theorem}[{\cite[Theorem 1.7]{DH22}}]\label{thm-bpf}
%Let $(X,B\ge 0)$ be a log pair, where $X$ is a normal $\mathbb{Q}$-factorial compact K\"ahler 3-fold.
%Let $\alpha\in H^{1,1}_{\textup{BC}}(X)$ be a nef class.
%Assume that one of the following conditions is satisfied:
%\begin{enumerate}
 %   \item[(i)] $(X,B)$ is klt and $\alpha-(K_X+B)$ is nef and big, or
  %  \item[(ii)] $(X,B)$ is dlt and $\alpha-(K_X+B)$ is a K\"ahler class.
%\end{enumerate}
%Then there exists a proper surjective morphism with connected fibres $\psi:X\to Z$ to a normal K\"ahler variety $Z$ with rational singularities and a K\"ahler class $\alpha_Z\in H_{BC}^{1,1}(Z)$ on $Z$ such that $\alpha=\psi^*\alpha_Z$.
%Moreover, in case (ii) we can choose $\psi$ to be a projective morphism.
%\end{theorem}

\begin{remark}[Differences with the previous paper]\label{rmk-difference}
Comparing the present paper with \cite{Zho22}, we give a few remarks on the differences of the approaches in the proofs.
\begin{enumerate}
\item As mentioned in \cite[Remark 1.6]{Zho22}, we were not able to run the $G$-MMP in \cite{Zho22}. 
One of the difficulty to show the $G$-equivariancy we met was the unknown of the finiteness of $(K_X+\xi)$-negative extremal rays when $\xi$ is a  nef and big (possibly non-rational) class.
Our method therein was to reduce $X$ to its minimal model $X_{\textup{min}}$ with the induced action $G|_{X_{\textup{min}}}$ consisting of bimeomorphic transformations.
Applying the trick of Albanese closure and the speciality of threefolds, we showed that the descended $G|_{X_{\textup{min}}}$ lies in $\textup{Aut}(X_{\textup{min}})$, i.e., the elements in $G|_{X_{\textup{min}}}$ are indeed biholomorphic.
However, the proof relies on the abundance for the case $\kappa(X)=0$. 
\item As kindly pointed out by Doctor Sheng Meng, one can fix the finiteness issue after applying the recent progress on the log minimal model theory for K\"ahler threefolds (cf.~\cite{DO22} and \cite{DH22}); see Proposition \ref{pro-nef-big-finite}. 
Hence, we are able to run a log minimal model program which is $G$-equivariant. 
Further, the full log version of the base point free theorem \cite[Theorem 1.7]{DH22} for K\"ahler threefolds provides us a powerful tool to achieve the end product of the $G$-MMP (cf.~Proposition \ref{prop-end-pdt}) and thus we finally confirm the existence of the $G$-MMP and prove Theorem \ref{thm-main}.  
\item In contrast to the previous proof of \cite[Theorem 1.3]{Zho22}, our present proof in this note does not rely on the incomplete case of the abundance any more.  
It only depends on the abundance for  numerically trivial log-canonical classes, which has been  confirmed in  \cite[Corollary 1.18]{CGP20} and \cite[Theorem 1.1 (3)]{DO22}. 
\end{enumerate}
\end{remark}
We end up this section by pointing out that the $G$-equivariancy of the minimal model program has its own interests, especially in the classification problem. 
The existence of the $G$-MMP indicates the fundamental building blocks of the automorphism groups. 

\subsubsection*{\textbf{\textup{Acknowledgments}}}
The author would like to deeply thank Professor De-Qi Zhang and Doctor Sheng Meng for many inspiring  discussions. 
The author would also like to thank Professor Andreas H\"oring for answering his question on the log minimal model program, and the referee for the very careful reading and suggestions to improve the paper. 
This work was supported by the Institute for Basic Science (IBS-R032-D1-2022-a00). 

\section{Preliminaries}
In this section, we unify some notations and prepare some results for the construction of the $G$-MMP. 
We follow the notations and terminologies in \cite[Section 2]{KM98}, \cite{HP16}, \cite{DH22} and \cite{Zho22} (cf.~\cite[Section 2]{Zho21}).
For the convenience of readers and us, we recall the following definition which will be crucially used in this paper.
\begin{definition}[{\cite[Section 3]{HP16}, \cite[Section 2]{DH22}}]
Let $(X,\omega)$ be a normal compact K\"ahler space, where $\omega$ is a fixed K\"ahler form, and $H_{\textup{BC}}^{1,1}(X)$ the Bott-Chern cohomology of real closed $(1,1)$-forms with local potentials, or closed $(1,1)$-currents with  local potentials.
Let $\textup{N}_1(X)$ be the vector space of real closed currents of bi-dimension $(1,1)$ modulo the equivalence: $T_1\equiv T_2$ if and only if $T_1(\eta)=T_2(\eta)$ for all real closed $(1,1)$-forms $\eta$ with local potentials. 
Let $\overline{\textup{NA}}(X)\subseteq \textup{N}_1(X)$ be the closure of the cone generated by the classes of positive closed currents of bi-dimension $(1,1)$.
\begin{enumerate}
    \item A positive closed $(1,1)$-current $T$ with local potentials is called a \textit{K\"ahler current}, if $T\ge \varepsilon\omega$ for some $\varepsilon>0$.
    A $(1,1)$-class $\alpha\in H_{\textup{BC}}^{1,1}(X)$ is called \textit{big} if it contains a K\"ahler  current.
    \item A $(1,1)$-class $\alpha\in H_{\textup{BC}}^{1,1}(X)$ is called \textit{nef} if it can be represented by a form $\eta$ with local potentials such that for every $\varepsilon$, there exists a $\mathcal{C}^{\infty}$ function $f_\varepsilon$ such that $\eta+dd^cf_\varepsilon\ge-\varepsilon\omega$.
    Denote by $\textup{Nef}(X)$ the cone generated by nef $(1,1)$-classes. 
    \item A $(1,1)$-class  $\alpha\in H_{\textup{BC}}^{1,1}(X)$ is called \textit{pseudo-effective},  if it can be represented by a positive closed $(1,1)$-current $T$ which is locally the form $dd^cf$ for some plurisubharmonic function $f$. 
    \item A  big class $\alpha$ is called \textit{modified K\"ahler} if it contains a K\"ahler current $T$ such that the Lelong number $\nu(T,D)=0$ for every prime divisor $D$ on $X$ (cf.~\cite[Definition 2.2]{Bou04}). 
\end{enumerate}
\end{definition}
\begin{notation}\label{not-xi}
Let $(X,D,G)$ be a  pair satisfying \hyperlink{Hypothesis}{Hypothesis}.
Let $\pi:\widetilde{X}\to X$ be a $G$-equivariant resolution (cf.~e.g.~\cite[Theorem 2.0.1]{Wlo09}).
Applying the proof of \cite[Theorems 4.3 and 4.7]{DS04} to the (lifted) action of $G$ on $\pi^*\textup{Nef}(X)$, there are three nef classes $\pi^*\xi_j$ with $j=1,2,3$ on $\widetilde{X}$ as common eigenvectors of $G$ such that $\xi_1\cdot\xi_2\cdot\xi_3\neq 0$.
To be more precise, there are three characters $\chi_j:G\to\mathbb{R}_{>0}$ such that $g^*(\pi^*\xi_j)=\chi_j(g)(\pi^*\xi_j)$ for each $g\in G$.
Since $X$ is a threefold, each $\xi_j$ is also nef 
(cf.~\cite[Lemma 3.13]{HP16}); hence there exist three nef common eigenvectors of $G$ on $\textup{Nef}(X)$.
Let $\xi:=\xi_1+\xi_2+\xi_3$.
Then $\xi^3>0$ and $\xi$ is thus a nef and big class on $X$ (cf.~\cite[Theorem 0.5]{DP04} and \cite[Proposition 2.6]{Zho21}). 
\end{notation}

The  proposition below plays a significant role in the proof of  Theorem \ref{thm-main} and has its own interests. 
It is well-known in the projective case by Kodaira's lemma and the cone theorem. 
In the analytic case however, the proof needs more arguments. 
\begin{proposition}\label{pro-nef-big-finite}
Let $X$ be a normal $\mathbb{Q}$-factorial compact K\"ahler threefold and $\Delta_0$ an effective $\mathbb{Q}$-divisor such that $(X,\Delta_0)$ has only klt singularities.  
Let $\alpha$ be a nef and big class on $X$. 
Suppose that one of the following conditions holds.
\begin{enumerate}
    \item $K_X+\Delta_0$ is pseudo-effective; or 
    \item $K_X+\Delta_0$ is not pseudo-effective but $K_X+\Delta_0+\alpha$ is a big class. 
\end{enumerate} 
Then there are only finitely many $(K_X+\Delta_0+\alpha)$-negative extremal rays, all of which are generated by rational curves.
\end{proposition}

Before proving the above proposition, we prepare the following lemma, which shares a similar proof as in the projective setting.
\begin{lemma}[{cf.~\cite[Corollary 1.19]{KM98}}]\label{lem-compact}
Let $X$ be a compact K\"ahler space with a fixed K\"ahler class $\omega$.
Suppose that $X$ has at worst rational singularities. 
Then the subset
 $S_M:=\{c\in\overline{\textup{NA}}(X)~|~\omega\cdot c\le M\}$ is  compact for every positive number $M$.
\end{lemma}
\begin{proof}
Let us fix a norm $||\cdot||$ on the finite dimensional vector space $\textup{N}_1(X)$, and assume that $S_M$ is not compact. 
Then there is a sequence $\{c_i\}\subseteq \overline{\textup{NA}}(X)$ such that $||c_i||\to\infty$ as $i\to\infty$. 
However, the normalized set $\{\frac{c_i}{||c_i||}\}$ is bounded; thus one may pick a subsequence $\{\frac{c_{i_k}}{||c_{i_k}||}\}$ which converges to a non-zero point $c_0\in\overline{\textup{NA}}(X)$ as $k\to\infty$.
So  $$(\omega\cdot c_0)=\lim_{k\to\infty}\frac{(\omega\cdot c_{i_k})}{||c_{i_k}||}=0,$$ 
 a contradiction to $\omega$ being a K\"ahler class (cf.~\cite[Proposition 3.15]{HP16}). 
\end{proof}

\begin{proof}[Proof of Proposition \ref{pro-nef-big-finite}]
Since $\alpha$ is  nef and big, by \cite[Lemma 2.36]{DH22}, there exist a modified K\"ahler class $\eta$ and an effective $\mathbb{Q}$-divisor $F$ such that $\alpha=\eta+F$ and $(X,\Delta:=\Delta_0+F)$ is klt.
We will show the finiteness of $(K_X+\Delta+\eta)$-negative extremal rays.

First, we assume (1), i.e., $K_X+\Delta_0$ is pseudo-effective. 
Then every $(K_X+\Delta_0+\alpha)$-negative extremal ray $R$ is $(K_X+\Delta_0)$-negative, and hence generated by a rational curve $R=\mathbb{R}_{\ge 0}[\ell]$  (see \cite[Theorem 10.12]{DO22} and  \cite[Theorem 4.2]{CHP16}).
Since $(X,\Delta)$ is also klt with $K_X+\Delta$ being pseudo-effective, %such $R$, once $(K_X+\Delta)$-negative, is also a $(K_X+\Delta)$-negative extremal ray.
by \cite[Theorem 10.12]{DO22}, there is a positive number $d$ such that every $(K_X+\Delta+\eta)$-negative extremal ray $R_i$ with $R_i=\mathbb{R}_{\ge}[\ell_i]$ satisfies $-(K_X+\Delta)\cdot \ell_i\le d$ (note that here, we include the case $(K_X+\Delta)\cdot \ell_i\ge 0$). 
Therefore,  $\eta\cdot \ell_i<-(K_X+\Delta)\cdot\ell_i\le d$.

We shall prove that, there are only finitely many such curve class with $\eta\cdot [\ell_i]<d$.
Indeed, since $\eta$ is modified K\"ahler, applying \cite[Proposition 2.3]{Bou04} and \cite[Page 990, Footnote  (5)]{CHP16} to $\eta$ (on the singular space $X$), we  have a suitable resolution $\pi:\widetilde{X}\to X$ such that $\pi^*\eta=\widetilde{\eta}+E$ where $\widetilde{\eta}$ is a K\"ahler class and $E$ is an effective $\pi$-exceptional $\mathbb{R}$-divisor on $\widetilde{X}$. 
Suppose to the contrary that there are infinitely many curve classes $[\ell_i]$ with $\eta\cdot [\ell_i]<d$.
Since $\pi(E)$ is the union of finitely many curves and points, we may assume that all of such curves $\ell_i$ are not contained in $\pi(E)$.
Let $\widetilde{\ell}_i$ be the proper transform of $\ell_i$ along $\pi$.
Now that different $\ell_i$ and $\ell_j$ lie in different numerical classes, their proper transforms $\widetilde{\ell}_i$ and $\widetilde{\ell}_j$ also lie in different numerical classes.
Moreover, $\widetilde{\eta}\cdot\widetilde{\ell}_i=\pi^*\eta\cdot\widetilde{\ell}_i-E\cdot\widetilde{\ell}_i\le\eta\cdot\ell_i\le d$.
So there are infinitely many curve classes $[\widetilde{\ell}_i]$ on $\widetilde{X}$ such that $\widetilde{\eta}\cdot\widetilde{\ell}_i\le d$.
However, by Lemma \ref{lem-compact}, $S_d$ is compact;  hence such curve classes (as a closed discrete set) in $S_d$ are finitely many, which produces a contradiction.
Therefore, if $K_X+\Delta_0$ is pseudo-effective, there are only finitely many $(K_X+\Delta+\eta)=(K_X+\Delta_0+\alpha)$-negative extremal rays.

Second, we assume (2), i.e., $K_X+\Delta_0$ is not pseudo-effective. 
Then $X$ is uniruled (cf.~\cite{Bru06}). 
Also, we may assume that the object of the MRC fibration of $X$ is a non-uniruled surface;  otherwise, $X$ is projective (cf.~\cite[Introduction]{HP15} or \cite[Lemma 2.39]{DH22}) and our proposition then follows from the usual Kodaira's lemma and the cone theorem. 
Since $K_X+\Delta+\eta$ is  big by assumption, for a general fibre $F$ of the MRC fibration,  $(K_X+\Delta+\eta)\cdot F>0$.
With the same argument as in the above paragraph after replacing \cite[Theorem 10.12]{DO22} by \cite[Corollary 4.10]{DH22}, we finish the proof.
\end{proof}

Given a nef and big class $\alpha$ on a normal compact K\"ahler space, we define the \textit{null locus} $\textup{Null}(\alpha)$  as the union of the proper subvarieties $V$ such that $\alpha|_V$ is not big (cf.~\cite{CT15}). 
\begin{lemma}[{cf.~\cite[Lemma 3.9]{Zha16}}]\label{lem-null-periodic}
With the same assumption and notations as in \hyperlink{Hypothesis}{Hypothesis} and Notation \ref{not-xi},  
the null locus of the nef and big class $\xi$ 
$$\textup{Null}(\xi):=\bigcup_{\xi|_V~\textup{not big}}V=\bigcup_{V~\textup{is}~ G\textup{-periodic}}V.$$
In particular, there are only finitely many $G$-periodic proper analytic subvarieties. 
\end{lemma}
\begin{proof}
For every closed subvariety $V$ on $X$, the restriction $\xi|_V$ is still nef.  (cf.~e.g.~\cite[Proposition 2.6]{Zho21}). 
Hence, $\xi|_V$ being not big is equivalent to $0=(\xi|_V)^{\dim V}=\xi^{\dim V}\cdot V$ (cf.~\cite[Theorem 0.5]{DP04} and \cite[Proposition 2.6]{Zho21}).
Applying a theorem of Collins and Tosatti \cite{CT15} to the pull-back of the class $\xi$ to a resolution of $X$, $\textup{Null}(\xi)$ has only finitely many irreducible components.

First, let $V$ (with $d:=\dim V$) be a proper $G$-periodic subvariety.
With $G$ replaced by a finite-index subgroup, we may assume $g^*V=V$ for any $g\in G$. 
On the one hand, by the projection formula, $\chi_1\cdot\chi_2\cdot\chi_3=1$.
On the other hand, for any $j_1,j_2$, 
 the image of $\varphi:G\to\mathbb{R}^2$ sending $g$ to $(\log\chi_{j_1}(g),\log\chi_{j_2}(g))$ is a spanning lattice of $(\mathbb{R}^2,+)$ in the sense that the image $\textup{Im}\varphi$ is discrete in the additive group $(\mathbb{R}^2,+)$ and hence $\textup{Im}\varphi\cong\mathbb{Z}^2$ (cf.~\cite[Section 4]{DS04}). 
As a result, for any $i_1,\cdots,i_d$ (with $d\le 2$), we can choose $g\in G$ such that $\chi_{i_k}(g)>1$ for every $i_k$.
By the projection formula, we have 
$$\chi_{i_1}(g)\cdots\chi_{i_d}(g)(\xi_{i_1}\cdots\xi_{i_d}\cdot V)=g^*\xi_{i_1}\cdots g^*\xi_{i_d}\cdot g^*V=\xi_{i_1}\cdots\xi_{i_d}\cdot V.$$
As each $\chi_{i_k}(g)>1$ by the choice of $g$, we have $\xi_{i_1}\cdots\xi_{i_d}\cdot V=0$ for any $i_1,\cdots,i_d$ (with $d\le 2$). 
Since $\xi=\sum\xi_i$ and $\dim (V)\ge 1$, we have $\xi^d\cdot V=0$. 
One inclusion is verified.

Second,  let $V$ (with $d:=\dim V$) be an irreducible closed subvariety such that $\xi^{d}\cdot V=0$.  
By the nefness of $\xi_i$, we have $\xi_{i_1}\cdots\xi_{i_d}\cdot V=0$ for all $i_k$.  
With the same reason as above, $\xi_{i_1}\cdots\xi_{i_d}\cdot g(V)=0$ and then  $\xi^{d}\cdot g(V)=0$ for any $g\in G$.
So $g(V)\subseteq\textup{Null}(\xi)$ and  there is a natural inclusion
$V\subseteq\overline{\cup_{g\in G}g(V)}\subseteq \textup{Null}(\xi)$. 
Since the closure $\overline{\cup_{g\in G}g(V)}$ is $G$-stable and contained in $\textup{Null}(\xi)$ which is a union of finitely many proper subvarieties, 
$\textup{Null}(\xi)$ is contained in the right hand side.
\end{proof}

In what follows, we  recall the following lemma from previous papers.% which will be used in the next section. 

\begin{lemma}[{\cite[Lemma 2.7]{Zho22}; cf.~\cite[Lemma 3.7]{Zha16}}]\label{lem-periodic}
With the same assumption and notations as in \hyperlink{Hypothesis}{Hypothesis} and Notation \ref{not-xi}, 
for every $G$-periodic $(k,k)$-class $\eta$ with $k=1$ or $2$, $\xi^{3-k}\cdot \eta=0$; in particular, $\xi^2\cdot c_1(X)=\xi\cdot c_1(X)^2=0$ and $\xi\cdot \widetilde{c_2}(X)=0$ (where $\widetilde{c_2}(X)$ is the orbifold second Chern class as defined in \cite[Section 5]{GK20}). 
\end{lemma}

To close this section, we prove the lemma below which is known to experts as an analytic version of Kodaira's lemma.
It will be used in the next section. 
\begin{lemma}\label{lem-reso-null-locus}
Let $X$ be a compact K\"ahler manifold and $\eta$ a nef and big $(1,1)$-class on $X$.
Then there is a bimeromorphic holomorphic map $f:X'\to X$ from another compact K\"ahler manifold $X'$ such that $f^*\eta=\omega+E$ where $\omega$ is some K\"ahler class, $E$ is an effective $\mathbb{R}$-divisor, and the support of its $f$-image $\textup{Supp}\,f(E)$ coincides with the null locus of $\eta$.
\end{lemma}
\begin{proof}
By \cite[Theorem 3.17 (ii)]{Bou04},
there is a K\"ahler current $T\in \eta$ with analytic singularities precisely along
the non-K\"ahler locus, i.e.,  $\textup{Sing}\,~T=E_{\textup{nK}}(\eta)$ (cf.~ \cite{Dem92} and \cite[Theorem 2.2]{CT15}). 
Note also that $E_{\textup{nK}}(\eta)$ coincides with 
 $\textup{Null}(\eta)$ (cf.~\cite[Theorem 1.1]{CT15}). 
After blowing up the associated coherent ideal sheaf (with its support along $E_{\textup{nK}}(\eta)$) and taking  Hironaka's
resolution of singularities (cf.~e.g.~\cite[Definition 2.5.1 and Proof of Proposition 2.3]{Bou04}), 
we get a bimeromorphic holomorphic map $f:X'\to X$ such that  $f^*\eta=\omega+E$ where $\omega$ is some K\"ahler class and $E$ is an effective $\mathbb{R}$-divisor with its $f$-image having the support along $\textup{Supp}\,f(E)=\textup{Null}(\eta)$. 
Our lemma is thus proved.
\end{proof}

\section{Equivariancy of the MMP, Proof of Theorem \ref{thm-main}}
In this section, we establish the $G$-equivariant minimal model program for K\"ahler threefolds and prove Theorem \ref{thm-main}.   
Throughout this section, we stick to \hyperlink{Hypothesis}{Hypothesis} and Notation \ref{not-xi}. 
Roughly speaking, when $X$ is non-uniruled,  the $G$-MMP  is to contract the null locus of $\xi$ so as to get a K\"ahler class with good properties (cf.~\cite{Zha16}). 

Let us begin with the proposition below, which reveals the end product of such $G$-MMP. 
In contrast to the projective setting,  the nef and big $(1,1)$-class $\xi$ here is possibly not  an $\mathbb{R}$-divisor.
As a consequence, the proof of the $G$-equivariancy of the fibration induced by the base point free theorem \cite[Theorem 1.7]{DH22} does not follow immediately from Kodaira's lemma and more arguments are needed (cf.~Claim \ref{claim-equivariancy}).

\begin{proposition}\label{prop-end-pdt}
Theorem \ref{thm-main} holds true if $K_X+D+\xi$ is already nef.
\end{proposition}

\begin{proof}[Proof of Proposition \ref{prop-end-pdt}]
In the view of \cite[Theorem 1.3 (1) and its proof]{Zho22}, we only need to consider the case when $X$ is not uniruled (cf.~\cite[Lemma 2.10]{Zha09}).
With $\xi$ replaced by $2\xi$ if necessary, we may assume that $K_X+D+\xi$ is also big. 

Since $K_X+D+\xi$ is nef, 
by the base point free theorem \cite[Theorem 1.7]{DH22},  there exist a proper surjective morphism with connected fibres $\psi:X\to Z$ onto a normal K\"ahler variety $Z$ and a K\"ahler class $\beta$ on $Z$ such that $K_X+D+\xi=\psi^*\beta$. 
By the bigness of $K_X+D+\xi$, $\psi$ is bimeromorphic.

\begin{claim}\label{claim-equivariancy}
$\psi$ is $G$-equivariant.
\end{claim}

\noindent
\textbf{Proof of Claim \ref{claim-equivariancy}:}
Let $R$ be a $(K_X+D+\xi)$-trivial extremal ray. 
If $R$ is $\xi$-positive, then taking a positive number $\varepsilon<1$, one has
$$(K_X+D+(1-\varepsilon)\xi)\cdot R<(K_X+D+\xi)\cdot R=0.$$
By Proposition \ref{pro-nef-big-finite},  there are only finitely many such $(K_X+D+\xi)$-trivial but $\xi$-positive extremal rays.
Therefore, after replacing $\xi$ by a  multiple, we may assume that  every $(K_X+D+\xi)$-trivial extremal ray  is $\xi$-trivial.  
By \cite[Theorem 1.3]{CHP16},  every $(K_X+D+\xi)$-trivial curve is thus $\xi$-trivial. 
We shall follow  \cite[Proof of Theorem 6.4]{DH22} to establish the $G$-equivariancy of $\psi$.  
%Let us briefly recall the construction of $\psi$ for the convenience of readers; see 

By \cite[Lemma 2.36]{DH22}, there is a decomposition  $\xi=\eta+F$ such that $\eta$ is modified K\"ahler and $(X,D+F)$ is still klt.
Let us first prove that the support of $F$ is $G$-periodic so that the pair $(X,D+F,G)$ still satisfies  \hyperlink{Hypothesis}{Hypothesis}.  
Let $\pi:\widetilde{X}\to X$ be a $G$-equivariant resolution with $\pi^*\xi$ nef and big. 
Applying Lemma \ref{lem-reso-null-locus}, 
%By \cite[Theorem 3.17 (ii)]{Bou04},
%there is a K\"ahler current $T\in \pi^*\xi$ with analytic singularities precisely along
%the non-K\"ahler locus, i.e.,  $\textup{Sing}\,~T=E_{\textup{nK}}(\pi^*\xi)$ (cf.~ \cite{Dem92} and \cite[Theorem 2.2]{CT15}).
%Note also that $E_{\textup{nK}}(\pi^*\xi)$ coincides with 
% $\textup{Null}(\pi^*\xi)$ (cf.~\cite[Theorem 1.1]{CT15}). 
%After blowing up the associated coherent ideal sheaf (with its support along $E_{\textup{nK}}(\pi^*\xi)$) and taking  Hironaka's
%resolution of singularities (cf.~e.g.~\cite[Definition 2.5.1 and Proof of Proposition 2.3]{Bou04}), 
we get a (not necessarily $G$-equivariant) bimeromorphic holomorphic map $\pi':\widetilde{X}'\to \widetilde{X}$ such that  $\pi'^*\pi^*\xi=\omega+E$ where $\omega$ is some K\"ahler class and $E$ is an effective $\mathbb{R}$-divisor with its $\pi'$-image having the support along $\textup{Supp}\,\pi'(E)=\textup{Null}(\pi^*\xi)$. 
By Lemma \ref{lem-null-periodic}, the support of $F$ (as the divisorial component of $\pi(\pi'(E))$; cf.~\cite[Lemma 2.36]{DH22}) is $G$-periodic. 
In the following, we can run $(K_X+D+\xi)$-trivial but $(K_X+D+F)$-negative minimal model program $G$-equivariantly. 

\setlength\parskip{3pt}\par \vskip 0.3pc 
\noindent
\textbf{Step 1.} First, we deal with the case when every $(K_X+D+\xi)$-trivial curve is $(K_X+D+F)$-non-negative. 
Then the null locus $\textup{Null}(K_X+D+\xi)$ is the union of finitely many curves (cf.~\cite[Paragraph 4, Page 48, Proof of Theorem 6.4]{DH22}). 
By \cite[Proposition 6.2]{DH22} (cf.~\cite[Theorem 4.2]{CHP16}), there is a proper bimeromorphic holomorphic map $\psi_0:X\to Z$ contracting  $\textup{Null}(K_X+D+\xi)$ to a single point.
Moreover, $K_X+D+\xi=\psi_0^*\beta$ for some K\"ahler class $\beta$ on $Z$ (cf.~\cite[Ending part of Proof of Theorem 6.4]{DH22}). 
For every curve $\ell$  contracted by $\psi_0$, it follows from the projection formula that $\xi_i\cdot g(\ell)=0$, noting that each $\xi_i$ is nef and $\ell$ is also $\xi$-trivial.
Then $0=(K_X+D+\xi)\cdot g(\ell)=\beta\cdot \psi_0(g(\ell))$, which implies that $g(\ell)$ is also contracted by $\psi_0$ for every $g\in G$. 
By the rigidity lemma (cf.~\cite[Lemma 4.1.13]{BS95}),  $G$ descends to $Z$ holomorphically, i.e., $\psi_0$ is $G$-equivariant. 

\setlength\parskip{3pt}\par \vskip 0.3pc 
\noindent
\textbf{Step 2.} Suppose in the following that there is a $(K_X+D+\xi)$-trivial extremal ray $R$, which is $\xi$-trivial but $(K_X+D+F)$-negative. 
Then it follows from \cite[Theorem 1.3]{CHP16} that such $R=\mathbb{R}_{\ge0}[C]$ is generated by a  curve $C$ with $-(K_X+D+F)\cdot C\le 4$. 
Since $\xi=\eta+F$, we have $\eta\cdot C\le 4$ and thus  there are only finitely many extremal rays $R$ such that $(K_X+D+\xi)\cdot R=0$, $\xi\cdot R=0$ but $(K_X+D+F)\cdot R<0$ (cf.~Proof of Proposition \ref{pro-nef-big-finite}). 
Since the support of $D+F$ is $G$-periodic, with $G$ replaced by a finite-index subgroup, the $G$-image of $R$ is also $(K_X+D+\xi)$-trivial, $\xi$-trivial but $(K_X+D+F)$-negative. 
With $G$ replaced by a finite-index subgroup again, $G$ fixes all of such extremal rays and  the contraction $\phi: X\to Y$ of  $R$ is thus $G$-equivariant (the existence of $\phi$ is due to \cite[Theorems 1.5 and 2.18]{DH22}; cf.~\cite[Section 10]{DO22}).
Furthermore, since $R$ is $\xi$-trivial, each $\xi_i=\phi^*\xi_{i,Y}$ for some nef $(1,1)$-class $\xi_{i,Y}$ on $Y$.\setlength\parskip{0pt} 

If $\phi$ is  divisorial, then  set $X_1:=Y$, $\xi_{X_1}:=\sum \xi_{i,Y}$, $\eta_1:=\phi_*\eta$, $F_1:=\phi_*F$ and $D_1:=\phi_*D$.
If $\phi$ is  small  with the flipped contraction $\phi^+:X^+\to Y$ still being $G$-equivariant by the choice of $X^+:=\textup{Proj}_Y\oplus_{m\ge 0}\phi_*\mathcal{O}_X(\lfloor m(K_X+D+F)\rfloor)$, then set  $X_1:=X^+$, $\xi_{X_1}:=\sum (\phi^+)^*\xi_{i,Y}$ and set $\eta_1, F_1, D_1$ to be the direct images of $\eta,F,D$ under the flip $X\dashrightarrow X^+$ respectively. 
Then we continue this procedure with the new pair $(X_1,D_1)$, noting that $\eta_1$ is still modified K\"ahler (cf.~\cite[Lemma 2.35]{DH22}) and the pair $(X_1,D_1+F_1,G)$ satisfies \hyperlink{Hypothesis}{Hypothesis}. 
By \cite[Theorem 1.1]{DH22}, this program will terminate in finitely many steps and thus we finally arrived at Step 1.

By the ending part of \cite[Proof of Theorem 6.4]{DH22},
the composite map $\psi:X\dashrightarrow Z$ is holomorphic and $(K_X+D+\xi)$-trivial. 
So  we finish the proof of Claim \ref{claim-equivariancy} by showing the $G$-equivariancy step by step. \qed

\setlength\parskip{6pt}\par \vskip 1pc

Now we come back to proving Proposition \ref{prop-end-pdt}. 
By \cite[Lemma 3.3]{HP16},  every nef $G$-eigenvector  $\xi_i$ coincides with $\psi^*\xi_{Z,i}$ and then we have $\xi=\psi^*(\xi_Z:=\sum\xi_{Z,i})$. 
Since $K_X+D=\psi^*(K_Z+D_Z)$ and $(X,D)$ is klt, the pair $(Z,D_Z:=\psi_*(D))$ is also klt. 
Further, $D_Z$ is also $G$-periodic.
\setlength\parskip{0pt}
%\begin{remark}
%We remark here that, there may exist some other positive closed $(2,2)$-currents $T$ (which is not generated by curves) such that $\psi_*T=0$.
%For example, let $S$ be an irreducible component of divisorial part of the null locus $\textup{Null}(\xi)$ such that $\psi(S)$ is a single point and let $\omega$ be any K\"ahler form on $X$.
%Then $S\cdot\omega$ is one of such positive closed $(2,2)$-current.
%\end{remark}
\begin{claim}\label{claim-nef-case}
$K_Z+D_Z\sim_{\mathbb{Q}}0$.
\end{claim}
Suppose Claim \ref{claim-nef-case} for the time being.
Then  $\xi_Z\sim_\mathbb{Q} K_Z+D_Z+\xi_Z=\beta$ is a K\"ahler class.
By Lemma \ref{lem-null-periodic}, there is no  positive dimensional $G$-periodic subvariety on $Z$ and hence $D_Z=0$. 
In particular, $K_Z\sim_\mathbb{Q}0$ and Theorem \ref{thm-main} (3) is proved. 
%In the following, we follow the proof as in \cite[Proposition 2.9]{Zho22}.

Take an integer $m$ such that $mK_Z\sim 0$. 
Let $\tau:Z':=\textbf{Spec}\,\oplus_{i=0}^{m-1}\mathcal{O}_Z(-iK_Z)\to Z$
be the global index one cover with $K_{Z'}=\tau^*K_Z\sim 0$.
Then $Z'$ has only canonical singularities (cf.~\cite[Proposition 5.20 and Corollary 5.24]{KM98}). 
Clearly, one can lift $G$ to $Z'$ via their actions on $K_Z$ and  $\xi_{Z'}:=\tau^*\xi_Z$  is still a K\"ahler class (cf.~\cite[Propsoition 3.5]{GK20}).
By \cite[Proposition 3.6]{GK20} and Lemma \ref{lem-periodic}, $\widetilde{c_2}(Z')\cdot \xi_{Z'}=0$. %where $\widetilde{c_2}$ is the  orbifold second Chern class as defined in \cite[Section 5]{GK20}. 
Therefore, $Z'$ is a $Q$-complex torus (cf.~\cite[Theorem 1.1]{GK20}). 
By the Galois closure trick (cf.~\cite[Lemma 7.4]{GK20}), the quotient  $Z$ is also a $Q$-complex torus. 
So Theorem \ref{thm-main} (1) is proved. 

Let $a:T\to Z$ be the Albanese closure (cf.~\cite[Definition 2.5]{Zho22}) such that the group $G$ (on $Z$) lifts to $G_T$ (on $T$) holomorphically.
Let $H:=\textup{Gal}(T/Z)$ be the galois group. 
Since each $g_T\in G_T$ normalizes $H$, the composite $a\circ g_T$ is $H$-invariant.
By the universality of the quotient morphism $a$ over the \'etale locus, $a\circ g_T$ factors through $a$.
Hence, the Albanese closure $a$ is $G$-equivariant with $G\cong G_T/H$.
In the view of Lemma \ref{lem-null-periodic}, the singular locus $\textup{Sing}\,Z$ consists of finitely many isolated points, noting that $\textup{Sing}\,Z$ is  $G$-periodic and $\xi_Z$ is K\"ahler (cf.~Lemma  \ref{lem-periodic}). 
So the group $H$ acts freely  outside finitely many isolated points by the purity of branch loci and  
 Theorem \ref{thm-main} (2) is proved.

\setlength\parskip{4pt}\par \vskip 0.3pc

\noindent
\textbf{Proof of Claim \ref{claim-nef-case} (End of Proof of Proposition \ref{prop-end-pdt}):} 
Since $Z$ is non-uniruled by our assumption in the beginning of the proof, $K_Z$ is pseudo-effective (cf.~\cite{Bru06}). 
%Besides, since $K_X+D=\psi^*(K_Z+D_Z)$, $K_Z+D_Z$ is also $\mathbb{Q}$-Cartier.
Let $m\in\mathbb{N}$ be the Cartier index of $K_Z+D_Z$ and take a $G$-equivariant resolution $\sigma:\widetilde{Z}\to Z$.
Then the null locus of $\xi_{\widetilde{Z}}:=\sigma^*\xi_Z$, which is the union of positive dimensional $G$-periodic proper subvarieties of $\widetilde{Z}$ (cf.~Lemma \ref{lem-null-periodic}),  coincides with the non-K\"ahler locus of $\xi_{\widetilde{Z}}$  (cf.~\cite[Theorem 1.1]{CT15}). 
By Lemma \ref{lem-periodic},  $\sigma^*(K_Z+D_Z)\cdot\xi_{\widetilde{Z}}^2=0$.
\setlength\parskip{0pt}

We show that, for every $G$-periodic subvariety $B\subseteq\widetilde{Z}$, the restriction $\xi_{\widetilde{Z}}|_{B}=0$. % in the sense that $\xi_{\widetilde{Z}}\cdot b=0$ for any $(2,2)$-current $b$ on $B$. 
By Lemma \ref{lem-periodic}, we may assume $\dim B=\dim\sigma(B)=2$, for otherwise, the triviality of $\xi_{\widetilde{Z}}|_B$ follows from that of $\xi_Z|_{\sigma(B)}$, noting that $\dim \sigma(B)\le 1$ and $\sigma(B)$ is also $G$-periodic.
%So we only need to deal with the case $\dim\sigma(B)=2$. 
Then, 
$\sigma^*(K_Z+D_Z+\xi_Z)|_B$ is nef and big.
Since both $\sigma^*(K_Z+D_Z)$ and $B$ are $G$-periodic, applying Lemma \ref{lem-periodic} once more, we have 
$$\sigma^*(K_Z+D_Z+\xi_Z)|_B\cdot \xi_{\widetilde{Z}}|_B=\sigma^*(K_Z+D_Z+\xi_Z)\cdot \xi_{\widetilde{Z}}\cdot B=0.$$
Thus, $(\xi_{\widetilde{Z}}|_B)^2\le 0$ and hence 
$(\xi_{\widetilde{Z}}|_B)^2=0$ by the nefness of  $\xi_{\widetilde{Z}}$. 
Applying  the Hodge index theorem and noting that $(\sigma^*(K_Z+D_Z+\xi_Z)|_B)^2>0$, we finally have $\xi_{\widetilde{Z}}|_B\equiv 0$.

Since $\xi_{\widetilde{Z}}$ is nef and big, by Lemma \ref{lem-reso-null-locus}, there is a bimeromorphic holomorphic map (which is not necessarily $G$-equivariant) 
$\sigma':\widetilde{Z}'\to \widetilde{Z}$ such that  $\sigma'^*\xi_{\widetilde{Z}}=\omega+E$ where $\omega$ is some K\"ahler class and $E$ is an effective $\mathbb{R}$-divisor with the support $\textup{Supp}\,\sigma'(E)=\textup{Null}(\xi_{\widetilde{Z}})$ (which is $G$-periodic).  
Then $(\sigma'^*\xi_{\widetilde{Z}})|_E=(\sigma'|_E)^*(\xi_{\widetilde{Z}}|_{\sigma'(E)})\equiv 0$ by the above paragraph. 
Therefore, 
\begin{equation}\label{eq1}\tag{\dag}
0=(\varphi:=\sigma\circ\sigma')^*(K_Z+D_Z)\cdot
(\varphi^*\xi_Z)^2=\varphi^*(K_Z+D_Z)\cdot(\omega+E)\cdot\omega.
\end{equation}
Furthermore, the nefness of the following class
$$\varphi^*(K_Z+D_Z)|_E=(\sigma'|_E)^*(\sigma^*(K_Z+D_Z)|_{\sigma'(E)})=(\sigma'|_E)^*(\sigma^*(K_Z+D_Z+\xi_Z)|_{\sigma'(E)}),$$ 
implies that $\varphi^*(K_Z+D_Z)\cdot E\cdot\omega\ge 0$.
Together with Equation (\ref{eq1}),  $\varphi^*(K_Z+D_Z)\cdot\omega^2=0$. 
%By \cite[Theorem 3.1]{DS04}, $\varphi^*(K_Z+D_Z)^2\cdot\omega\le 0$ and the equality holds if and only if $\varphi^*(K_Z+D_Z)\equiv 0$. 
%However, for  $e\gg 1$,  $\varphi^*(K_Z+D_Z)+e\omega$ is also a K\"ahler class.
Since $K_Z+D_Z$ is pseudo-effective by noting that $Z$ is non-uniruled (cf.~\cite{Bru06}), 
%By the pseudo-effectivity of $\varphi^*(K_Z+D_Z)$, one has $$0\le\varphi^*(K_Z+D_Z)\cdot(\varphi^*(K_Z+D_Z)+e\omega)\cdot\omega=\varphi^*(K_Z+D_Z)^2\cdot\omega\le 0.$$
we have $\varphi^*(K_Z+D_Z)\equiv 0$ as classes. 
As $\dim Z=3$, we have $K_Z+D_Z\equiv 0$ (cf.~\cite[Proposition 3.14]{HP16}). 
So the abundance for numerically trivial pairs gives us  $K_Z+D_Z\sim_{\mathbb{Q}}0$ (cf.~\cite[Corollary 1.18]{CGP20} or \cite[Theorem 1.1]{DO22}).  %or \cite[Corollary 1.18]{CGP20}).
We finish the proof of Claim \ref{claim-nef-case}.
\end{proof}

\begin{proof}[Proof of Theorem \ref{thm-main}]
In the view of \cite[Theorem 1.3 (1)]{Zho22} and its proof, we may assume that $X$ is not uniruled. 
In the following, let us run the $G$-equivariant minimal model program for the pair $(X,D+\xi)$ (cf.~\hyperlink{Hypothesis}{Hypothesis} and Notation \ref{not-xi}).
If $K_X+D+\xi$ is already nef, then we are done by Proposition \ref{prop-end-pdt}.

We consider the case when $K_X+D+\xi$ is not nef. 
Then there is a $(K_X+D+\xi)$-negative extremal ray $R$.
By Proposition \ref{pro-nef-big-finite}, such  $(K_X+D+\xi)$-negative extremal rays are finitely many. 
Therefore, with $\xi$ replaced by a multiple, we may assume that all of such extremal rays are $\xi$-trivial (and hence $\xi_i$-trivial for $i=1,2,3$). 
If $R=\mathbb{R}_{\ge 0}[\ell]$ is one of such ray with  $(K_X+D+\xi)\cdot\ell<0$ and $\xi\cdot\ell=0$, then $\xi_i\cdot g(\ell)=0$ for each $i$ and thus $\xi\cdot g(\ell)=0$.
So $(K_X+D+\xi)\cdot g(\ell)<0$ and  $g_*R$ is also one of such  $(K_X+D+\xi)$-negative extremal rays. 
With $G$ replaced by a finite-index subgroup, we may assume that $R$ is $G$-stable. 
Let $\phi: X\to Y$ be the contraction of $R$ (cf.~\cite[Theorems 1.5 and 2.18]{DH22}, \cite[Section 10]{DO22}). 
Then $G$ descends to a biregular action on $Y$ and each $\xi_i=\phi^*\xi_{Y,i}$ for some nef common $G|_Y$-eigenvector $\xi_{Y,i}$ (cf.~\cite[Lemma 3.3]{HP16}).

If $\phi$ is  divisorial, with $(X,D,G)$ replaced by $(Y,D_Y:=\phi(D),G|_Y)$, we come back  and continue this program, noting that  $(Y,D_Y:=\phi(D),G|_Y)$ satisfies \hyperlink{Hypothesis}{Hypothesis}; see \cite[Theorem  1.1]{DH22}. 
If $\phi$ is  small, then  \cite[Theorem 4.3]{CHP16} confirms the existence of the flip $\phi^+:X^+\to Y$. 
Indeed, $X=\textup{Proj}_Y\oplus_{m\ge 0}\phi_*\mathcal{O}_X(\lfloor-m(K_X+D_X)\rfloor)$ and $X^+=\textup{Proj}_Y\oplus_{m\ge 0}\phi_*\mathcal{O}_X(\lfloor m(K_X+D_X)\rfloor)$.
Clearly, $G$ descends to $Y$ and  lifts to $X^+$ holomorphically.
Then with $(X,D,G)$ replaced by $(X^+,D^+:=(\phi^+)^{-1}(\phi(D)), G|_{X^+})$, we come back  and continue this program, noting that we pull back $\xi_{Y,i}$ to $X^+$ to get new nef  common $G|_{X^+}$-eigenvectors. 
Similarly, $(X^+,D^+:=(\phi^+)^{-1}(\phi(D)), G|_{X^+})$ satisfies \hyperlink{Hypothesis}{Hypothesis}; see~\cite[Theorem 1.1]{DH22}.

By \cite[Theorem 1.1]{DH22}, such ($G$-equivariant) log minimal model program will terminate after finitely many steps.
So we finally arrive at the model $X_n$ with  nef $K_{X_n}+D_n+\xi_n$ and  get the fibration $X_n\to Z$ by the base point free theorem (cf.~\cite[Theorem 1.7]{DH22}).
By Proposition \ref{prop-end-pdt}, we are remained to prove that, the $G$-equivariant bimeromorphic composite map $X\dashrightarrow Z$ is indeed holomorphic.
Suppose to the contrary that  $X\dashrightarrow Z$ is not holomorphic.
Then there exists some flip $X_i\dashrightarrow X_{i+1}:=X_i^+$ over $Y$ in the $G$-MMP, such that $Y\dashrightarrow Z$ is not holomorphic.
Let $\phi_i:X_i\to Y$  be the corresponding flipping  contraction of the extremal ray $R_i$ (with the flipped contraction $\phi_i^+:X_{i+1}\to Y$).
By the rigidity lemma (cf.~\cite[Lemma 4.1.13]{BS95}), there is some curve $C\in R_i$ such that $(\phi_i^+)^{-1}(\phi_i(C))$ is not contracted by $X_i^+\dashrightarrow Z$. 
But then, the image of $(\phi_i^+)^{-1}(\phi_i(C))$ on $Z$ is a $G$-periodic curve, contradicting the first assertion of Theorem \ref{thm-main} (3). 
So we finish the proof of Theorem \ref{thm-main}.
\end{proof}

\end{document}